\documentclass[graybox]{svmult}

\usepackage{mathptmx}       
\usepackage{helvet}         
\usepackage{courier}        
\usepackage{type1cm}        
\usepackage{makeidx}         
\usepackage{graphicx}        
\usepackage{multicol}        
\usepackage[bottom]{footmisc}

\usepackage{amssymb}
\usepackage{amsmath}
\usepackage{epstopdf}
\usepackage{color}

\makeindex             

\begin{document}

\title*{An Invariant-region-preserving (IRP) Limiter to DG Methods for Compressible Euler Equations}
\authorrunning{An IRP limiter for compressible Euler equations}
\author{Yi Jiang and Hailiang Liu}

\institute{Yi Jiang \at Department of Mathematics, Iowa State University, Ames, IA, 50010, USA. \\
  \email{yjiang1@iastate.edu}
\and Hailiang Liu \at Department of Mathematics, Iowa State University, Ames, IA, 50010, USA. \\
\email{hliu@iastate.edu}}
%
%
\maketitle


\abstract{We introduce an explicit invariant-region-preserving limiter applied to DG methods for compressible Euler equations. The invariant region considered consists of positivity of density and pressure and a maximum principle of a specific entropy. The modified polynomial by the limiter preserves the cell average, lies entirely within the invariant region and does not destroy the high order of accuracy for smooth solutions. Numerical tests are presented to illustrate the properties of the limiter. In particular, the tests on Riemann problems show that the limiter helps to damp the oscillations near discontinuities.}
\section{Introduction}
\label{sec:1}
We consider the one dimensional version of the compressible Euler equations for the perfect gas in gas dynamics:
\begin{equation}\label{IVP}
\begin{aligned}
&\vec{w}_t+F(\vec{w})_x=0, \quad t>0, \; x\in \mathbb{R},\\
&\vec{w}=(\rho,m,E)^\top, \quad F(\vec{w})=(m,\rho u^2+p,(E+p)u)^\top
\end{aligned}
\end{equation}
with 
\begin{equation}
m=\rho u, \; E=\frac{1}{2}\rho u^2+\frac{p}{\gamma -1},
\end{equation}
where $\gamma>0$ is a constant ($\gamma =1.4$ for the air), $\rho$ is the density, $u$ is the velocity, $m$ is the momentum, $E$ is the total energy and $p$ is the pressure; supplemented by initial data $\mathbf{w}_0(x)$. For the associated entropy function $s=\log\left(\frac{p(x)}{\rho ^{\gamma}(x)} \right)$,  it is known that 
\begin{equation}
A=\{(\rho, m, E)^\top, \rho>0, \;p >0, \; s\geq s_0\}
\end{equation}
for any $s_0\in \mathbb{R}$ is an invariant region in the sense that if $\mathbf{w}_0(x) \in A$, then $\mathbf{w}(x, t) \in A$ for all $t>0$ (see e.g. \cite{Sm83, Hoff85}).  At numerical level this set is proved to be invariant by the first order Lax-Friedrichs scheme (see \cite{Frid01}), and by the first order Finite Element method (see \cite{GP16}), in which a  larger class of hyperbolic conservation laws is considered. It is difficult, if not impossible, to preserve such set by a high order numerical method unless some nonlinear limiter is imposed at each step while marching in time. In this work we design such a limiter.

In recent years an interesting mathematical literature has developed devoted to high order maximum-principle-preserving schemes 
for scalar conservation equations (see \cite{ZhangShu2010a}) and positivity-preserving schemes for hyperbolic systems including compressible Euler equations (see e.g. \cite{PerthameShu1996, ZhangShu2010b, ZhangXiaShu2012}). In \cite{PerthameShu1996} up to third order positivity-preserving finite volume schemes are constructed based on positivity-preserving properties by the corresponding first order schemes for both density and pressure of one and two dimensional compressible Euler equations. Following \cite{PerthameShu1996}, positivity-preserving high order DG schemes for compressible Euler equations were first introduced in  \cite{ZhangShu2010b}, where  the limiter in \cite{ZhangShu2010a} is generalized. A recent work by Zhang and Shu in \cite{ZhangShu2012} introduced a minimum-entropy-principle-preserving limiter for high order schemes to the compressible Euler equation. In their work, the limiter for entropy part is enforced separately from the ones for the density and pressure and is given implicitly with the limiter parameter solved by Newton's iteration. 

For the isentropic gas dynamics, the invariant region is bounded by two global Riemann invariants; for which the authors have designed an explicit limiter in \cite{JL16} to preserve the underlying invariant region, called an invariant-region-preserving (IRP)  limiter.  Our goals in this work are to design an IRP limiter for the compressible Euler system (\ref{IVP}) and to rigorously prove that such a limiter does not destroy the high order accuracy in general cases. Our limiter differs from that in \cite{ZhangShu2012} in two aspects: (i) it is given in an explicit form; (ii) the scaling reconstruction depends on a uniform parameter for the whole vector solution polynomial; in addition to the rigorous proof of the preservation of the accuracy by the limiter.  As a result, the limiter preserves the positivity of density and pressure and also a maximum principle for the specific entropy \cite{Tadmor}, with reduced computational costs  in numerical implementations.
\section{The Limiter}
\label{sec:2}
 We construct a novel limiter based on both the cell average (strictly in $A$) and the high order polynomial approximation, 
 which is not entirely in $A$;  through a linear convex combination as  in \cite{ZhangShu2010a, ZhangShu2012}.
\subsection{Averaging is a Contraction}
\label{subsec:2.1} For initial density $\rho _0>0$ and pressure $p_0>0$, we fix
\begin{equation}\label{s0}
s_0=\inf _x\log \left(\frac{p_0(x)}{\rho ^{\gamma}_0(x)} \right),
\end{equation}
and define $q=(s_0-s)\rho$, then the set $A$ is equivalent to the following set: 
\begin{equation}
\Sigma =\{\mathbf{w}:\quad \rho >0, \; p>0, \; q\leq 0 \},
\end{equation}
which is convex due to the concavity of $p$ and convexity of $q$.  By using set $\Sigma$ we are able to work out an explicit limiter which has the invariant-region-preserving property.   Numerically, the set of admissible states is defined as
\begin{equation}
\Sigma^{\epsilon}=\{\vec{w}: \quad  \rho \geq \epsilon, \; p \geq \epsilon, \; q \leq 0\},
\end{equation}
with its interior denoted by 
\begin{equation}
{\Sigma}^{\epsilon}_0=\{\vec{w}: \quad \rho > \epsilon, \; p> \epsilon, \; q<0\},
\end{equation}
where $\epsilon$ is a small positive number chosen (say as $10^{-13}$ in practice) so that $q$ is well defined.  

For any bounded interval $I$ (or bounded domain in multi-dimensional case), we define the average of $\vec{w}(x)$ by 
\begin{equation}
\bar{\vec{w}}=\frac{1}{|I|} \int_I \vec{w}(x)dx
\end{equation}
where $|I|$ is the measure of $I$.  Such an averaging operator is a contraction:  
\begin{lemma}\label{lem: cvgIn}
Let $\vec{w}(x)$ be non-trivial vector polynomials. If $\vec{w}(x)\in {\Sigma}^{\epsilon}$ for all $x\in I$, then $\bar{\vec{w}}\in \Sigma^{\epsilon}_0$ for any bounded interval $I$.
\end{lemma}
\begin{proof}
For the entropy part, since $q$ is convex, using Jensen's inequality and the assumption, we have
\begin{equation}
q(\bar{\vec{w}})=q\left(\frac{1}{|I|}\int _I \vec{w}(x)dx \right)\leq \frac{1}{|I|}\int _Iq(\vec{w}(x))dx\leq 0.
\end{equation}
With this we can show $q(\bar{\vec{w}})<0$. Otherwise, if $q(\bar{\vec{w}})=0$, we must have $q(\vec{w}(x))=0$ for all $x \in I$; that is
\begin{equation}
(s_0-s(\bar{\vec{w}}))\bar{\rho}=(s_0-s(\vec{w}(x)))\rho(x).
\end{equation}
By taking average of this relation over $I$ on both sides, we have for $g_1=s\rho$,
\begin{equation}
s_0\bar{\rho}-g_1(\bar{\vec{w}})=s_0\bar{\rho}-\frac{1}{|I|}\int _Ig_1(\vec{w}(x))dx.
\end{equation}
This gives
\begin{equation}
\frac{1}{|I|}\int _Ig_1(\vec{w}(x))dx=g_1(\bar{\vec{w}}).
\end{equation}
By taking the Taylor expansion around $\bar{\vec{w}}$, we have
\begin{equation}
g_1(\vec{w}(x))=g_1(\bar{\vec{w}})+\triangledown _{\vec{w}}g_1(\bar{\vec{w}})\cdot \xi+\xi ^\top H_1\xi, \quad \forall x\in I, \quad \xi :=\vec{w}(x)-\bar{\vec{w}},
\end{equation}
which upon integration yields $\frac{1}{|I|}\int _I\xi ^TH_1\xi dx=0$, where $H_1$ is the Hessian matrix of $g_1$. This combined with the strict concavity of $g_1$ ensures that $\vec{w}(x)\equiv \bar{\vec{w}}$, which contradicts the assumption.

We can show $p(\bar{\vec{w}})>\epsilon$ by a similar contradiction argument. The density part is obvious.
\end{proof}

\subsection{Reconstruction}
\label{subsec:2.2}
Let $\vec{w}_h(x)=(\rho _h(x),m_h(x),E_h(x))^\top$ be a vector of polynomials of degree $k$ over an interval $I$, which is a high order approximation to the smooth function $\vec{w}(x)=(\rho(x),m(x),E(x))^\top \in {\Sigma}^{\epsilon}$.   We  assume that the average $\bar{\vec{w}}_h\in {\Sigma}^{\epsilon}_0$, but $\vec{w}_h(x)$ is not entirely located in ${\Sigma}^{\epsilon}$ for $x\in I$, then we can use the average as a reference in the following  reconstruction 
\begin{equation}\label{bb}
\tilde{\vec{w}}_h(x)=\theta \vec{w}_h(x)+(1-\theta)\bar{\vec{w}}_h,
\end{equation}
where 
\begin{equation}\label{limiter}
\theta = \min \{1, \theta _1, \theta _2,\theta _3\},
\end{equation}
with 
\begin{equation*}
\theta _1=\frac{\bar{\rho}_h-\epsilon}{\bar{\rho}_h-\rho _{h,\min}}, \quad \theta _2=\frac{p(\bar{\vec{w}}_h)-\epsilon}{p(\bar{\vec{w}}_h)-p_{h,\min}}, \quad \theta _3=\frac{-q(\bar{\vec{w}}_h)}{q_{h,\max}-q(\bar{\vec{w}}_h)}
\end{equation*}
and
\begin{equation}\label{aa}
\rho _{h,\min}=\min _{x\in I}\rho _h(x), \quad p_{h,\min}=\min _{x\in I}p(\vec{w}_h(x)), \quad q_{h,\max}=\max _{x\in I}q(\vec{w}_h(x)).
\end{equation}
Note that $p(\bar{\vec{w}}_h)>p_{h,\min}$ and $q(\bar{\vec{w}}_h)<q_{h, \max}$ due to the concavity of $p$ and convexity of $q$. Therefore $\theta _i's$ are well-defined and positive, for $i=1,2,3$. We can prove that this reconstruction has three desired properties, summarized in the following.
\begin{theorem}
\label{BigThm}
The reconstructed polynomial $\tilde{\mathbf{w}}_h(x)$ satisfies the following three properties:
\begin{itemize}
\item[(i)] $\quad$   the average is preserved, i.e. $\bar{\mathbf{w}}_h=\bar{\tilde{\mathbf{w}}}_h$;
\item[(ii)] $\quad$  $\tilde{\mathbf{w}}_h(x)$  lies entirely within invariant region ${\Sigma}^{\epsilon},  \forall x\in I$;
\item[(iii)] $\quad$  order of accuracy is maintained, i.e.,  $\| \tilde{\mathbf{w}}_h-\mathbf{w}\| _{\infty} \leq C \| \mathbf{w}_h-\mathbf{w}\| _{\infty}$,  provided $\| \mathbf{w}_h-\mathbf{w}\| _{\infty}$ is sufficient small, 
where $C$ is a positive constant that only depends on $\bar{\mathbf{w}}_h, \mathbf{w}$, and the invariant region $\Sigma^{\epsilon}$.
\end{itemize}
\end{theorem}
\begin{proof} (i) Since $0<\theta \leq 1$ is a uniform constant, average preservation is obvious.  \\
(ii) If  $\rho _{h,\min}\geq\epsilon$, $p_{h,\min}\geq\epsilon$, and $q_{h,\max}\leq 0$, then $\theta=1$, no reconstruction is needed. 
When $\theta =\theta _1$, we have
\begin{equation}
\begin{aligned}
\tilde{\rho}_h(x)=&\theta _1\rho _h(x)+(1-\theta _1)\bar{\rho}_h\\
=&(\bar{\rho}_h-\epsilon)\frac{\rho _h(x)-\rho _{h,\min}}{\bar{\rho}_h-\rho _{h,\min}}+\epsilon \geq \epsilon.
\end{aligned}
\end{equation}
Since $\theta _1\leq \theta _2$, we have $(p(\bar{\mathbf{w}}_h)-p_{h,\min})\theta _1+\epsilon\leq p(\bar{\mathbf{w}}_h)$. Therefore, by the concavity of $p$, we have
\begin{equation}
\begin{aligned}
p(\tilde{\mathbf{w}}_h)\geq &\theta _1p(\mathbf{w}_h)+(1-\theta _1)p(\bar{\mathbf{w}}_h)\\
=& \theta _1(p(\mathbf{w}_h)-p(\bar{\mathbf{w}}_h))+p(\bar{\mathbf{w}}_h)\\
\geq &\theta _1(p(\mathbf{w}_h)-p(\bar{\mathbf{w}}_h))+(p(\bar{\mathbf{w}}_h)-p_{h,\min})\theta _1+\epsilon\\
=& \theta _1(p(\mathbf{w}_h)-p_{h,\min})+\epsilon\geq \epsilon.
\end{aligned}
\end{equation}
For entropy part, since $\theta _1\leq \theta _3$, we have $\theta _1(q_{h,\max}-q(\bar{\mathbf{w}}_h))\leq -q(\bar{\mathbf{w}}_h)$. Therefore, by the convexity of $q$, we have
\begin{equation}
\begin{aligned}
q(\tilde{\mathbf{w}}_h)< &\theta _1q(\mathbf{w}_h)+(1-\theta _1)q(\bar{\mathbf{w}}_h)\\
=&\theta _1(q(\mathbf{w}_h)-q(\bar{\mathbf{w}}_h))+q(\bar{\mathbf{w}}_h)\\
\leq & \theta _1(q_{h,\max}-q(\bar{\mathbf{w}}_h))+q(\bar{\mathbf{w}}_h)\leq 0.
\end{aligned}
\end{equation}
In the case that $\theta =\theta _2$ or $\theta_3$ the proof is similar. \\
(iii) We prove for the case $\theta=\theta _2$, the other cases are similar.  In such case we only need to prove  
\begin{equation}\label{acc+}
\| \tilde{\mathbf{w}}_h-\mathbf{w}_h\| _{\infty} \leq C \| {\mathbf{w}}_h-\mathbf{w}\| _{\infty},
\end{equation}
from which (iii) follows by using the triangle inequality. Here and in what follows $\|\cdot\|_\infty:=\max_{x\in I}|\cdot|$.  We prove (\ref{acc+}) in four steps. 

Step 1. From (\ref{bb}) it follows that
\begin{equation}
\begin{aligned}
\| \tilde{\mathbf{w}}_h - \mathbf{w}_h\| _{\infty}
=&(1-\theta _2)\| \bar{\mathbf{w}}_h-\mathbf{w}_h\|_{\infty} \\
=&\frac{\max \limits _{x\in I} |\bar{\mathbf{w}}_h-\mathbf{w}_h(x)|}{p(\bar{\mathbf{w}}_h)-p_{h,\min}}(\epsilon -p _{h,\min}).
\end{aligned}
\end{equation}

Step 2. The overshoot estimate. Since $\mathbf{w}(x)\in \Sigma ^{\epsilon}$,
\begin{equation}
\epsilon -p_{h,\min}\leq \max _x(p(\mathbf{w})-p(\mathbf{w}_h))\leq C_1\| \mathbf{w}-\mathbf{w}_h\| _{\infty}, \quad C_1:=\| \triangledown p\| _{\infty}.
\end{equation}

Step 3. We map $I$ to $[0, 1]$ by $\xi=(x-a)/(b-a)$ for $I=[a, b]$, and  let $l_{\alpha}(\xi)$ ($\alpha =1,\cdots, N$) be the Lagrange interpolating polynomials at quadrature points $\hat \xi^\alpha \in [0, 1]$ with $N=k+1$, then $\mathbf{w}_h(x)-\bar{\mathbf{w}}_h=\sum ^N_{\alpha =1}(\mathbf{w}_h(\hat{x}^{\alpha})-\bar{\mathbf{w}}_h)l_{\alpha}(\xi)$, where  $\hat{x}^{\alpha}=a+(b-a)\hat \xi^\alpha$.  Hence, we have
\begin{equation}
\begin{aligned}
\max \limits _{x\in I} |\bar{\mathbf{w}}_h-\mathbf{w}_h(x)|\leq &\max _{\xi \in [0, 1]}\sum \limits^N_{\alpha =1}|l_{\alpha}(\xi)||\bar{\mathbf{w}}_h-\mathbf{w}_h(\hat{x}^{\alpha})|\\
\leq &C_2\max _{\alpha}|\bar{\mathbf{w}}_h-\mathbf{w}_h(\hat{x}^{\alpha})|, 
\end{aligned}
\end{equation}
where $C_2=\Lambda _{k+1}([0, 1])\doteq\max \limits_{\xi \in [0, 1]}\sum \limits^N_{\alpha =1}|l_{\alpha}(\xi)|$ is the Lebesgue constant.  Note that
\begin{equation}
\max \limits_{\alpha}|\bar{\mathbf{w}}_h-\mathbf{w}_h(\hat{x}^{\alpha})|
\leq \max \limits_{\alpha}|\bar{\rho}_h-\rho _h(\hat{x}^{\alpha})|
  +  \max \limits_{\alpha}|\bar{m}_h-m_h(\hat{x}^{\alpha})|
  +  \max \limits_{\alpha}|\bar{E}_h-E_h(\hat{x}^{\alpha})|.
\end{equation}
Define
\begin{equation}
\hat f _{h,\min}\doteq \min \limits_{\alpha}f (\mathbf{w}_h(\hat{x}^{\alpha})),\quad \hat{f} _{h,\max}\doteq \max \limits_{\alpha}f(\mathbf{w}_h(\hat{x}^{\alpha})), 
\end{equation}
we can show that
\begin{equation}
\max \limits_{\alpha}|\bar f_h-f_h(\hat{x}^{\alpha}))| 
\leq \max \{ \bar{f}_h-\hat{f}_{h,\min},\hat{f}_{h,\max}-\bar{f}_h \}\leq C_3 (\bar f_h -\hat f_{h, min}),
\end{equation}
where 
\begin{equation}
C_3=\max \left\lbrace 1,  \frac{1-\min \limits _{\alpha}\hat{w}_{\alpha}}{\min \limits_{\alpha}\hat{w}_{\alpha}}\right\rbrace.
\end{equation}
Here  $f_h=\rho _h, m_h, E_h$. The type of estimates using $C_2$ and $C_3$ is known, see \cite[Lemma 7, Appendix C]{Zhang17}), where the proof was accredited to Mark Ainsworth. \\

Step 4. The above three steps lead to 
\begin{equation}\label{eq:step4}
\| \tilde{\mathbf{w}}_h-\mathbf{w}_h\| _{\infty} \leq  C_1C_2 C_3  \frac{B}{p(\bar{\mathbf{w}}_h)-p _{h,\min}} \| \mathbf{w}-\mathbf{w}_h\| _{\infty},
\end{equation}
with 
\begin{equation}
B= \bar{\rho}_h - \hat {\rho}_{h, min} +\bar m_h -\hat m_{h, min} +\bar E_h -\hat E_{h, min}.
\end{equation}
On one hand, we have $p_{h, min} \leq \epsilon$ since $\theta =\theta _2\leq 1$, leading to 
\begin{equation}
p(\bar{\mathbf{w}}_h)-p _{h,\min} \geq p(\bar{\mathbf{w}}_h) -\epsilon;
\end{equation}
On the other hand the assumption $\theta =\theta _2\leq \theta _1$ implies 
\begin{equation}
\bar {\rho}_h -\hat {\rho}_{h, min} \leq \left( \frac{\bar{\rho}_h -\epsilon}{p(\bar{\mathbf{w}}_h)-\epsilon}\right) \cdot 
\left({p(\bar{\mathbf{w}}_h)-p _{h,\min}}\right). 
\end{equation}
By the assumption on the smallness of $\| \mathbf{w}_h-\mathbf{w}\| _{\infty}$ we have 
\begin{equation}
\bar m_h -\hat m_{h, min} \leq {2\| m-m_h\| _{\infty}+\bar{m}-m_{\min}}
\end{equation}
and 
\begin{equation}
\bar E_h -\hat E_{h, min} \leq \bar{E}+1.
\end{equation}
where $E \geq \frac{\epsilon}{\gamma-1}$ is used.  Collecting the above estimates we take 
\begin{equation}
C_4=\frac{\bar{\rho}_h-\epsilon+2\| m-m_h\|_{\infty}+\bar{m}-m_{\min}+\bar{E}+1}{p(\bar{\mathbf{w}}_h)-\epsilon}
\end{equation}
to conclude the desired estimate in (iii) with $C=\Pi_{i=1}^4 C_i$. 
\end{proof}

\subsection{Algorithm}
\label{subsec:2.3}
Let $\vec{w}^n_h$ be the numerical solution generated from a high order scheme of an abstract form
\begin{equation}\label{abstsche}
\vec{w}^{n+1}_h=\mathcal{L}(\vec{w}^n_h),
\end{equation}
where $\vec{w}^n_h=\vec{w}^n_h(x)\in V_h$, which is a finite element space of piecewise polynomials of degree $k$ over each computational cell $I$. Assume $\lambda =\frac{\Delta t}{h}$ is the mesh ratio, where $h$ is the characteristic length of the mesh size.

Provided that scheme (\ref{abstsche}) has the following property: there exists $\lambda _0$, and a test set $S_I$ in each computational cell $I$ such that if
\begin{equation}\label{suff}
\lambda \leq \lambda _0 \quad {\rm and} \quad  \vec{w}^n_h\in {\Sigma}^{\epsilon}, \quad x\in S_I
\end{equation}
then
\begin{equation}\label{prop}
\bar{\vec{w}}^{n+1}_h\in \Sigma^{\epsilon}_0,
\end{equation}
then the IRP limiter can be applied with $I$ replaced by $S_I$ in (\ref{aa}), i.e., 
\begin{equation}
\label{defqp}
\rho _{h, \min} = \min _{x\in S_I}\rho _h(x),\quad p_{h, \min} = \min_{x\in S_I } p(\mathbf{w}_h(x)), \quad q_{h, \max} =\max _{x\in S_I}q(\mathbf{w}_h(x)).
\end{equation}
Our algorithm is given as follows: 

\textbf{Step 1.} Initialization: take the piecewise $L^2$ projection of $\vec{w}_0$ onto $V_h$, such that
\begin{equation}
\int_I (\vec{w}^0_h(x)-\vec{w}_0(x))\phi(x)dx =0, \quad \forall \phi \in V_h.
\end{equation}
Also from $\vec{w}_0$, we compute $s_0$ as defined in (\ref{s0}) to determine the invariant region $\Sigma^{\epsilon}$.\\

\textbf{Step 2.} Impose the modified limiter (\ref{bb}), (\ref{limiter}) with (\ref{defqp}) on $\vec{w}^n_h$  for $n=0, 1, \cdots $.

\textbf{Step 3.} Update by the scheme:
\begin{equation}
\vec{w}^{n+1}_h=\mathcal{L}(\tilde{\vec{w}}^n_h).
\end{equation}
Return to \textbf{Step 2}.
\begin{remark}
Indeed the limiter (\ref{bb}), (\ref{limiter}) with (\ref{defqp}) can well enhance the efficiency of computation, and  we will use this modified IRP limiter in the numerical experiments. Note that with (\ref{defqp}), (i) and (iii) in Theorem \ref{BigThm} remain valid, and the 
resulting reconstructed polynomial lies within invariant region $\Sigma ^{\epsilon}$ only for $x\in S_I$.  
\end{remark}
\begin{remark} Notice that Lemma \ref{lem: cvgIn} ensures that $\bar{\vec{w}}^0_h$ lies strictly within $\Sigma^{\epsilon}_0$, therefore the limiter is valid already at the initialization step. 
\end{remark}
\begin{remark}Some sufficient conditions for (\ref{suff}) to ensure the cell average propagation property (\ref{prop}) for the DG method have been obtained for one-dimensional case (\cite{ZhangShu2010b}), as well as for rectangular meshes (\cite{ZhangShu2010b, ZhangShu2012}) and triangular meshes (\cite{ZhangXiaShu2012}) in two-dimensional cases. For example, 
the test set $S_I$ and the CFL condition given  in \cite[Theorem 2.1]{ZhangShu2010b} is
\begin{equation}
S_I=\{\hat{x}^{\alpha}, \alpha =1,\cdots ,N \},
\end{equation}
which is a set of $N$-point Legendre Gauss-Lobatto quadrature on $I$ with $2N-3\geq k$, and 
\begin{equation}
\lambda \| (|u|+c)\| _{\infty}\leq \frac{1}{2}\hat{w}_1,
\end{equation}
where $\hat{w}_1$ is the first Legendre Gauss-Lobatto quadrature weights for the interval $[-\frac{1}{2},\frac{1}{2}]$ such that $\sum ^N_{\alpha}\hat{w}_{\alpha}=1$.
\end{remark}

\section{Numerical Tests}
\label{sec:4}
We present numerical tests for the IRP limiter applied to a general high order DG scheme with the Lax-Friedrich numerical flux, using a proper time discretization. The semi-discrete DG scheme we take is a closed ODE system of the form 
\begin{equation}
\frac{d}{dt}\vec{W}=L(\vec{W}),
\end{equation}
where $\vec{W}$ consists of the unknown coefficients of the numerical solution in terms of the spatial basis, and $L$ is the corresponding spatial operator.

We consider the following two types of time discretizations. 
\begin{itemize}
\item[-] The third order SSP Runge-Kutta (RK3) method in \cite{ShuOsher1988} reads as
\begin{equation}
\begin{aligned}
\vec{W}^{(1)}=&\vec{W}^n+\Delta tL(\vec{W}^n)\\
\vec{W}^{(2)}=&\frac{3}{4}\vec{W}^n+\frac{1}{4}\vec{W}^{(1)}+\frac{1}{4}\Delta tL\left(\vec{W}^{(1)}\right)\\
\vec{W}^{n+1}=&\frac{1}{3}\vec{W}^n+\frac{2}{3}\vec{W}^{(2)}+\frac{2}{3}\Delta tL\left(\vec{W}^{(2)}\right).
\end{aligned}
\end{equation}
\item[-]The third order SSP multi-stage (MS) method in \cite{ShuTVD1988} reads as
\begin{equation}
\vec{W}^{n+1}=\frac{16}{27}(\vec{W}^n+3\Delta tL(\vec{W}^n))+\frac{11}{27}\left(\vec{W}^{n-3}+\frac{12}{11}\Delta tL(\vec{W}^{n-3})\right).
\end{equation}
\end{itemize} 
We apply the limiter at each time stage or each time step.

\begin{remark}
In the implementation of the third order SSP multi-step method, we apply SSP RK3 method in the first three time evolutions to obtain the starting values.
\end{remark}

In all of the following examples $\gamma =1.4$ is taken.

\runinhead{Example 1.}\textit{Accuracy Test}\\
We first test the accuracy of the IRP DG scheme. The initial condition is 
\begin{equation}
\rho _0(x)=1+\frac{1}{2}\sin(2\pi x),\quad u_0(x)=1, \quad p_0(x)=1. 
\end{equation}
The domain is $[0,1]$ and the boundary condition is periodic. The exact solution is
\begin{equation}
\rho (x,y,t)=1+\frac{1}{2}\sin(2\pi (x-t)), \quad u(x,t)=1, \quad p(x,t)=1.
\end{equation}
The results presented in Tables \ref{tb:p2WL} and \ref{tb:p3WL} show that using IRP limiter does not destroy high order accuracy.

\begin{table}[htbp]
\caption{}
\label{tb:p2WL}
\centering
\begin{tabular}{|c|cccc||cccc|}
\hline
$P^2$ DG&\multicolumn{4}{|c||}{SSP RK}  &\multicolumn{4}{|c|}{SSP multi-step}
\\ \hline
N &$L^{\infty}$Error &Order &$L^1$Error & Order &$L^{\infty}$Error &Order &$L^1$Error & Order\\
\hline
8    &5.43E-04 &/    &5.77E-04 &/    &5.35E-04 &/    &5.70E-04 &/ \\
16   &8.98E-05 &2.60 &8.55E-05 &2.75 &8.89E-05 &2.59 &8.53E-05 &2.74 \\
32   &1.04E-05 &3.11 &1.09E-05 &2.98 &1.03E-05 &3.11 &1.08E-05 &2.99 \\
64   &1.33E-06 &2.97 &1.40E-06 &2.95 &1.34E-06 &2.94 &1.39E-06 &2.95 \\
128  &1.67E-07 &2.99 &1.75E-07 &3.00 &1.75E-07 &2.94 &1.76E-07 &2.98 \\
\hline
\end{tabular}
\end{table}


\begin{table}[htbp]
\caption{}
\label{tb:p3WL}
\centering
\begin{tabular}{|c|cccc||cccc|}
\hline
$P^3$ DG&\multicolumn{4}{|c||}{SSP RK}  &\multicolumn{4}{|c|}{SSP multi-step}
\\ \hline
N &$L^{\infty}$Error &Order &$L^1$Error & Order &$L^{\infty}$Error &Order &$L^1$Error & Order\\
\hline
8    &1.44E-05 &/    &1.09E-05 &/    &1.42E-05 &/    &1.08E-05 &/ \\
16   &1.39E-06 &3.37 &7.23E-07 &3.92 &1.37E-06 &3.37 &7.07E-07 &3.94 \\
32   &7.06E-08 &4.30 &6.14E-08 &3.56 &6.93E-08 &4.31 &5.99E-08 &3.56 \\
64   &6.34E-09 &3.48 &3.18E-09 &4.27 &6.21E-09 &3.48 &3.03E-09 &4.30 \\
128  &3.50E-10 &4.18 &2.12E-10 &3.91 &3.30E-10 &4.23 &1.97E-10 &3.94 \\
\hline
\end{tabular}
\end{table}
In the following examples, we solve (\ref{IVP}) subject to several different Riemann initial data. We compare the numerical solution obtained from the DG scheme with IRP limiter (\ref{bb}), (\ref{limiter}) with (\ref{defqp}) and the one obtained from the DG scheme with only positivity-preserving limiter, that is, using $\theta =\min\{1,\theta _1,\theta _2\}$, where $\theta _1$ and $\theta _2$ are defined as in (\ref{limiter}).

\runinhead{Example 2.}\textit{Lax Shock Tube Problem}\\
Consider the Lax initial data:
\begin{equation}
(\rho,m,E)=
\begin{cases}
(0.445,0.311,8.928), \quad x<0, \\
(0.5, 0, 1.4275), \quad x>0,
\end{cases}
\end{equation}
which induces a composite wave, a rarefaction wave followed by a contact discontinuity and then by a shock. We calculate the exact solution by following the formulas given in \cite[Section 14.11]{FVLeVeque}. The $P^2$-DG scheme with SSP RK3 method in time discretization is tested on $N=100$ cells over $x\in [-2,2]$ at final time $T=0.5$. From Fig. \ref{fig:Lax}, we see that the IRP limiter helps to damp oscillations near the discontinuities.

\begin{figure}[htbp]
\caption{Lax shock tube problem. Exact solution (solid line) vs numerical solution (dots); Top: With positive-preserving limiter; Bottom: With IRP limiter}
\label{fig:Lax}
\centering
\includegraphics[scale=.65]{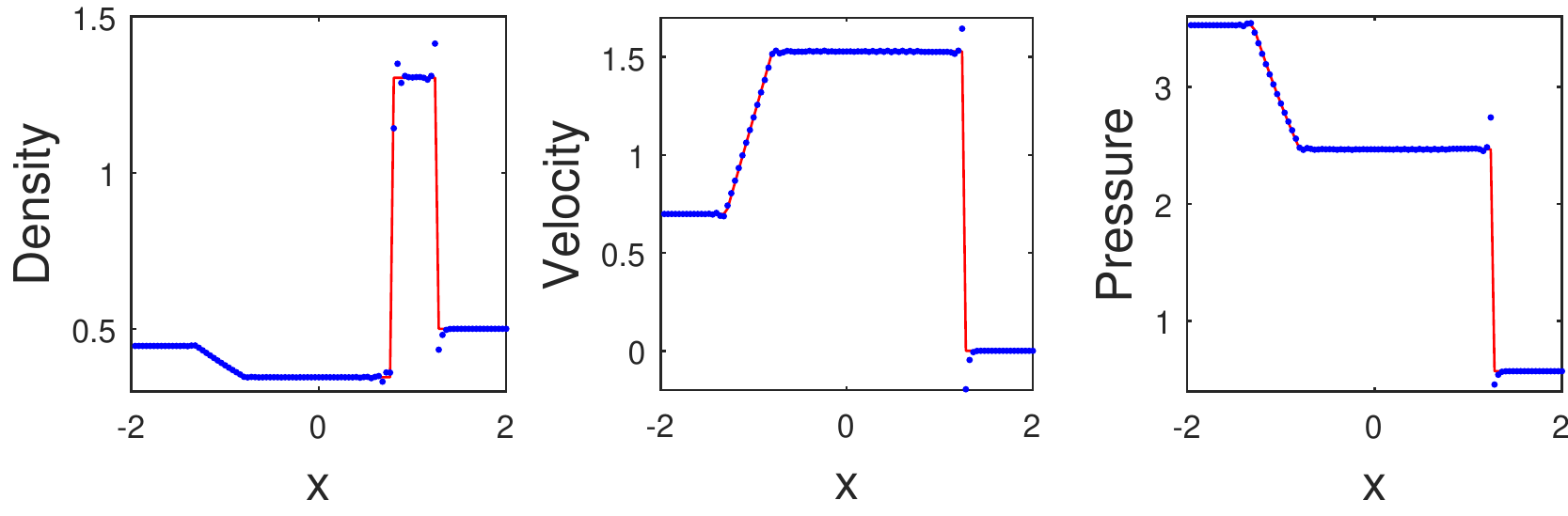}\\
\includegraphics[scale=.65]{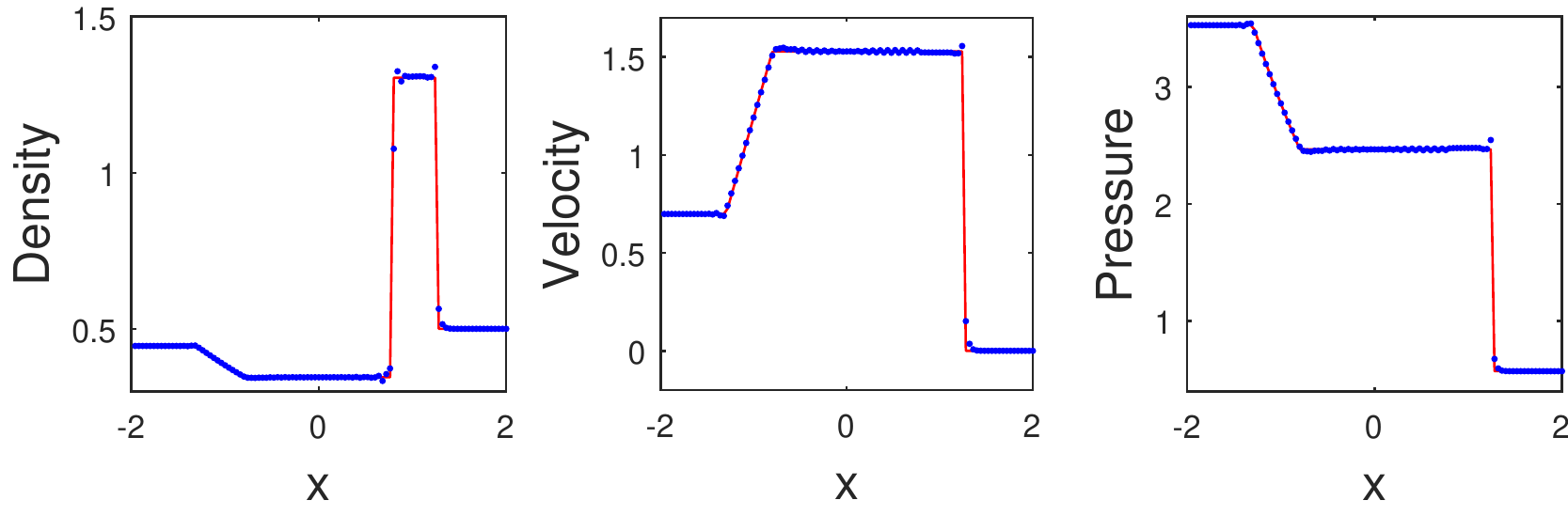}
\end{figure}

\runinhead{Example 3.}\textit{Shu-Osher Shock Tube Problem}\\
Consider the Shu-Osher problem:
\begin{equation}
(\rho,u,p)=
\begin{cases}
(3.857143, 2.629369, 10.3333), \quad x<-4,\\
(1+0.2\sin5x, 0,1), \quad x\geq -4.
\end{cases}
\end{equation}
The $P^2$-DG scheme with SSP RK3 method in time discretization is tested on $N=100$ cells over $x\in [-5,5]$ at final time $T=1.8$. The reference solution is obtained from $P^2$-DG scheme with SSP RK3 method on $N=2560$ cells. The results presented in Fig. \ref{fig:ShuOsher} show that the shock is captured well. 

\begin{figure}[htbp]
\caption{Shu-Osher problem. Exact solution (solid line) vs numerical solution (dots); Left: With positive-preserving limiter; Right: With IRP limiter}
\label{fig:ShuOsher}
\centering
\includegraphics[scale = .4]{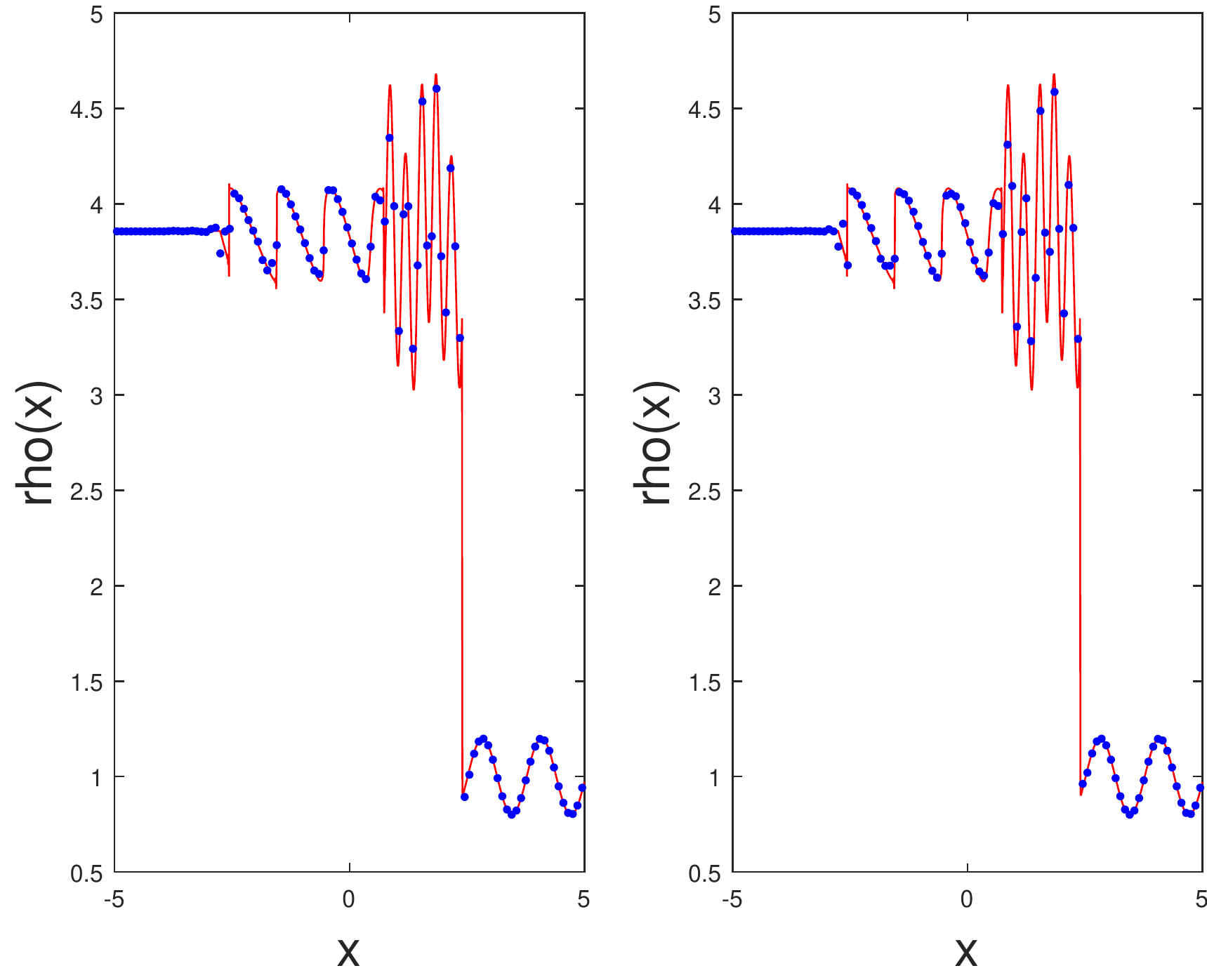}
\end{figure}

\section{Conclusion and Future Work}
\label{sec:5}
In this work, we introduced a novel IRP limiter for the one-dimensional compressible Euler equations. The limiter is made so that the reconstructed polynomial preserves the cell average, lies entirely within the invariant region and does not destroy the original high order of accuracy for smooth solutions.  Moreover, this limiter is explicit and easy for computer implementation. 
Let us point out that the IRP limiter (\ref{bb}) may be applied to multi-dimensional compressible Euler equations as well if we replace 
$I$ in (\ref{aa}) by multi-dimensional cells or test set in each cell.  Implementation details are in a forthcoming paper.  Future work would be to investigate IRP limiters for more general hyperbolic systems or specific systems in important applications.  

\begin{acknowledgement}This work was supported in part  by the National Science Foundation under Grant DMS1312636 and by NSF Grant RNMS (KI-Net) 1107291. 
\end{acknowledgement}

%
%

\begin{thebibliography}{99.}

\bibitem{Frid01}
H. Frid.
\newblock{\em Maps of Convex Sets and Invariant Regions for Finite-Difference Systems of Conservation Laws}. 
\newblock{ Archive for rational mechanics and analysis}, 160(3): 245-269, 2001.

\bibitem{FVLeVeque}
Randall J. LeVeque.
\newblock{\em Finite volume methods for hyperbolic problems (Vol. 31)}. 
\newblock{Cambridge university press}, 2002.

\bibitem{GP16}
J.-L. Guermond and B. Popov.
\newblock {\em Invariant domains and first-order continuous finite element approximation
for hyperbolic systems}, 
\newblock SIAM J. Numer. Anal., 54(4): 2466--2489, 2016.
\bibitem{Hoff85}
D. Hoff.  
\newblock  {\em  Invariant regions for systems of conservation laws.}  
\newblock Trans. Amer. Math. Soc., 289(2):591--610, 1985.

\bibitem{JL16}
Y. Jiang and H. Liu.
\newblock {\em An invariant-region-preserving limiter for the DG method to isentropic gas dynamics.}
\newblock preprint, 2016.

\bibitem{PerthameShu1996}
B. Perthame and C.-W. Shu.
\newblock{\em On positivity preserving finite volume schemes for Euler equations}. 
\newblock {Numerische Mathematik}, 73: 119--130, 1996.

\bibitem{ShuTVD1988}
C.-W. Shu.
\newblock{\em Total-variation-diminishing time discretizations}.
\newblock{SIAM J. Scientific Stat. Comput.}, 9: 1073--1084, 1988.

\bibitem{ShuOsher1988}
C.-W. Shu and S. Osher.
\newblock{\em Efficient implementation of essentially non-oscillatory shock-capturing schemes}.
\newblock{J. Comput. Phys.}, 77: 439--471, 1988.

\bibitem{Sm83}
J. Smoller.  
\newblock {\em Shock waves and reaction-diffusion equations.}
\newblock volume 258 of Grundlehren der Mathematischen Wissenschaften. Springer-Verlag, New York-Berlin, 1983.

\bibitem{Tadmor}
E. Tadmor.
\newblock {\em A minimum entropy principle in the gas dynamics equations.}
\newblock{Applied Numerical Mathematics, 2}, 211--219,1986.


\bibitem{ZhangShu2010a}
X. Zhang and C.-W. Shu.
\newblock{\em On maximum-principle-satisfying high order schemes for scalar conservation laws}.
\newblock{Journal of Computational Physics}, 229: 3091--3120, 2010.

\bibitem{ZhangShu2010b}
X. Zhang and C.-W. Shu.
\newblock{\em On positivity preserving high order discontinuous Galerkin schemes for compressible Euler equations on rectangular meshes}.
 \newblock{Journal of Computational Physics}, 229: 8918--8934, 2010.
 
\bibitem{ZhangShu2012}
X. Zhang and C.-W. Shu.
\newblock{\em  A minimum entropy principle of high order schemes for gas dynamics equations.}
\newblock {Numerische Mathematik}, 121:545--563, 2012.
 
\bibitem{ZhangXiaShu2012}
X. Zhang, Y. Xia and C.-W. Shu.
\newblock{\em Maximum-principle-satisfying and positivity-preserving high order discontinuous Galerkin schemes for conservation laws on triangular meshes}.
\newblock{Journal of Scientific Computing}, 50: 29--62,  2012.
 \bibitem{Zhang17} 
 X. Zhang. 
 \newblock {\em  On positivity-preserving high order discontinuous Galerkin schemes for compressible Navier-Stokes equations.} \newblock {Journal of Computational Physics},  328: 301Ð343, 2017.
 
\end{thebibliography}
%
\biblstarthook{}

\end{document}